\documentclass[11pt,a4paper]{article}
\usepackage{graphicx}
\usepackage{epsfig}
\usepackage{amssymb}
\usepackage{amsmath}
\usepackage{amsfonts}
\usepackage[latin1]{inputenc}
\usepackage{fancyhdr}
\usepackage{lastpage}
 \usepackage{float}

\font\tenbi=cmmib10
\font\sevenbi=cmmib10 at 7pt
\newfam\bifam
\textfont\bifam = \tenbi
\scriptfont\bifam= \sevenbi
\scriptscriptfont\bifam= \sevenbi

\numberwithin{equation}{section}
\numberwithin{figure}{section}
\numberwithin{table}{section}

\font\tendb=msbm10
\font\sevendb=msbm7
\font\fivedb=msbm5
\newfam\dbfam
\textfont\dbfam = \tendb
\scriptfont\dbfam= \sevendb
\scriptscriptfont\dbfam= \fivedb

\font\twlrsfs=rsfs10
\font\egtrsfs=rsfs7
\font\sixrsfs=rsfs5
\newfam\rsfsfam
\textfont\rsfsfam=\twlrsfs
\scriptfont\rsfsfam=\egtrsfs
\scriptscriptfont\rsfsfam=\sixrsfs

\allowdisplaybreaks

\newcommand{\titlesize}{\fontsize{15}{20pt}\bfseries}
\newcommand{\titulo}[1]{\vspace*{10pt}\begin{center}
               \titlesize{#1}\par\vspace*{10pt}\normalfont}

\newcommand{\direccion}[1]{\vspace*{12pt}\footnotesize
                \textit{#1}\end{center}\par\vspace*{20pt}}
\newtheorem{theorem}{Theorem}[section]
\newtheorem{lemma}{Lemma}[section]

\newtheorem{remark}{Remark}[section]

\newcommand{\cqfd}{\mbox{}\nolinebreak\hfill\rule{2mm}{2mm}\medskip\par}

\newenvironment{proof}[1] {\par\noindent{\bf Proof. }{#1}}{\cqfd}
\def\f {{\boldsymbol f}}
\def\d {{\boldsymbol d}}

\def\w {{\boldsymbol w}}
\def\v {{\boldsymbol v}}

\def\x {{\boldsymbol x}}
\def\n {{\boldsymbol n}}

\def\m {{\boldsymbol m}}

\def\u {{\boldsymbol u}}
\def\uu{{\boldsymbol u}}
\def\U {{\boldsymbol U}}
\def\H {{\boldsymbol H}}
\def\L {{\boldsymbol L}}

\def\p {{\phi}}

\def\R {\mathbb{R}}

\def\Om {\Omega}
\def\na {\nabla}
\def\eps {\varepsilon}

\def \hueco{\noalign{\medskip}}
\def \pato{\forall\,}
\def \beq{\begin{equation}}
\def \eeq{\end{equation}}
\def \ba{\begin{array}}
\def \ea{\end{array}}
\def \dis{\displaystyle}

\newtheorem{definition}[theorem]{Definition}

\title{{\sc Approximation of Smectic-A liquid crystals}}

\author{{\sc Francisco Guill\'en-Gonz\'alez}\thanks{Departamento de Ecuaciones Diferenciales y An\'alisis Num\'erico, Universidad de Sevilla,
Aptdo.~1160, 41080 Seville, Spain, email: {\tt guillen@us.es}}
\quad
{\sc Giordano Tierra}\thanks{Mathematical Institute, Faculty of Mathematics 
and Physics, Charles University, Prague 8, 186 75, Czech Republic, 
e-mail: {\tt gtierra@karlin.mff.cuni.cz}}
}

\begin{document}
\date{}
\maketitle

\begin{abstract}

In this paper, we present energy-stable numerical schemes for a
Smectic-A liquid crystal model. This model involve the hydrodynamic
velocity-pressure macroscopic variables $(\u,p)$ and the microscopic
order parameter of  Smectic-A liquid crystals, where its molecules
have a uniaxial orientational order and a positional order by
layers of normal and unitary vector $\n$.

We start from the formulation given in \cite{E} by using the so-called
layer variable $\phi$ such that $\n=\nabla \phi$ and the level sets of $\phi$ describe the layer structure of the Smectic-A liquid crystal. 
Then,  a strongly non-linear
parabolic system is derived coupling velocity and pressure unknowns of the Navier-Stokes equations $(\u,p)$
with a fourth order parabolic equation for $\phi$.

We will give a reformulation as a mixed second order problem
which let us to define some new energy-stable 
numerical schemes, by using second order finite differences in time
and $C^0$-finite elements in space.
Finally, numerical simulations are presented for $2D$-domains, showing the evolution of the system until it reachs an equilibrium configuration.

Up to our knowledge, there is not any previous numerical analysis
for this type of models.

\end{abstract}




%
\section{Introduction}
The topic of Liquid Crystals (LCs) is a multidisciplinary field
related to Chemistry, Engineering, Biology, Medicine, Physics and Mathematics. 
Usually, the discovery of liquid crystals is attributed to botanist F. Reinitzer who in 1888 found a substance that appeared to have two different melting points. A year later, O. Lehmann solved the problem with the description of a new state of matter midway between a liquid and a crystal.  In 1922, G. Friedel spoke for the first time about mesophases and in 1991 Pierre-Gilles de Gennes was awarded with Nobel Prize in Physics for his contributions related to LCs, in particular, for discovering that ''methods developed for studying order phenomena in simple systems can be generalized to more complex forms of matter, in particular to liquid crystals and polymers".
LCs are a state of matter that can be viewed as intermediate phases between solids and isotropic fluids.
Indeed, macroscopically such materials may flow like fluids but at microscopic scale their molecules have orientational properties (due to elasticity effects) and they can experience deformations as
elastic solids, hence LCs can be considered as anisotropic liquids.

Furthermore, LCs can be divided into thermotropic and lyotropic phases, where thermotropic phases change their state as the temperature is varying while lyotropic phase  change of state as concentration is varying.

Examples of LCs can be widely found in the natural world and in technological applications. For instance, most contemporary electronic devices use LCs for their displays and lyotropic liquid-crystalline phases can be found in living systems, forming proteins and cell membranes.

Within thermotropic LCs can be distinguished into two main different phases: Nematic and Smectic (see Figure~\ref{fig:LCS} for schematic description of Nematic and Smectic phases). 
In Nematic phases, the rod-shaped molecules have no positional order, but molecules self-align to have a long-range directional order with their long axes roughly parallel. Thus, the molecules are free to flow and their center of mass positions are randomly distributed as in a liquid, although they still maintain their long-range directional order. Moreover, most Nematics are uniaxial, i.e., they have one axis that is longer and preferred.

In Smectic phases (which are found at lower temperatures than the Nematic ones) molecules form well-defined layers that can slide over one another, i.e., Smectics are thus positionally ordered along one direction and are liquid-like within the layers. In particular, in Smectic-A phases, molecules are oriented along the normal vector of the layer, while in Smectic-C phases they are tilted away from the normal vector of the layer. We refer the reader to \cite{Collings} for further information on the physics and properties of the different LCs that can be found in the nature.

\begin{figure}[H]
\begin{center}
\includegraphics[width=0.8\textwidth]{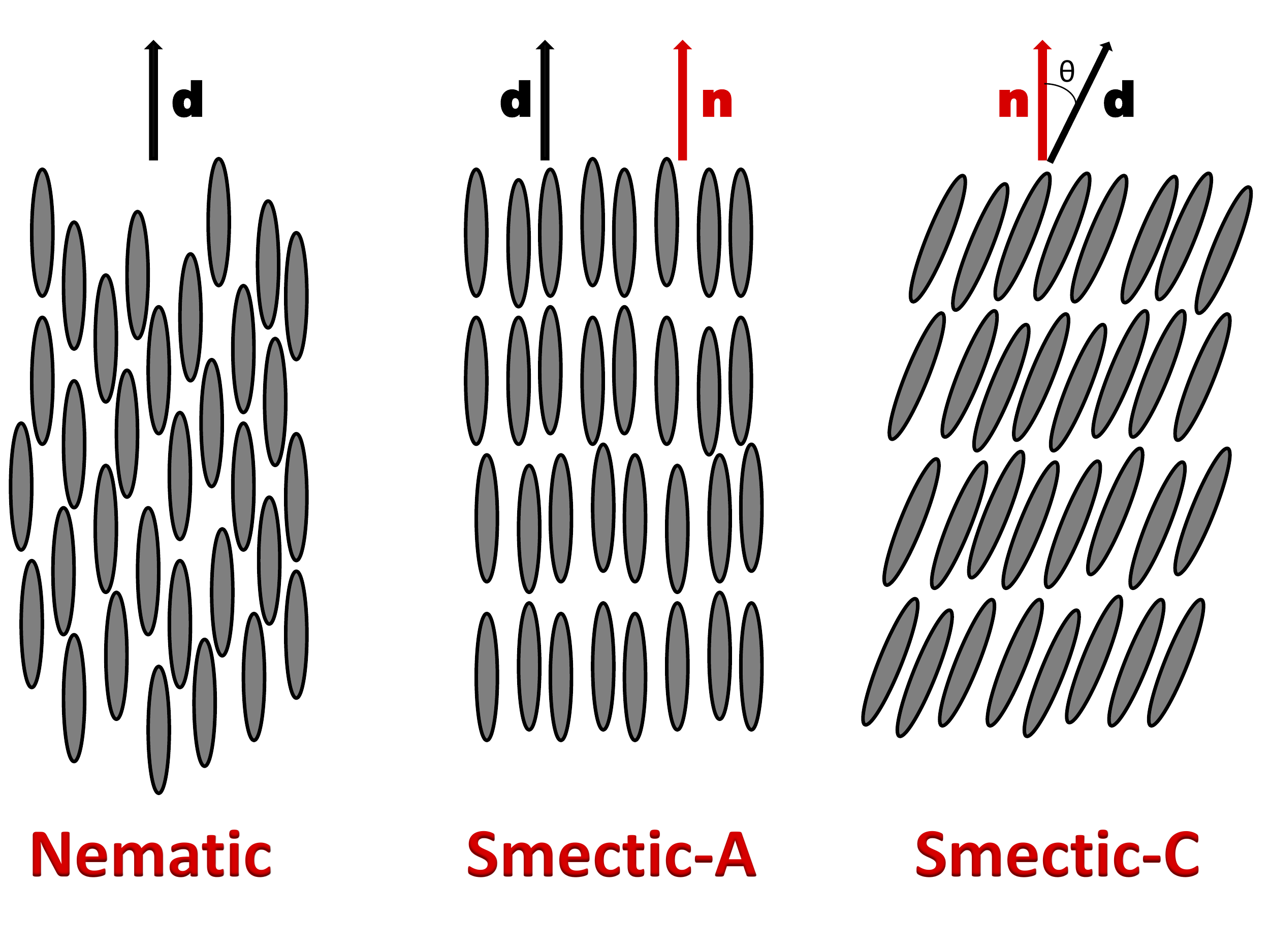}
\caption{Different phases of Liquid Crystals}\label{fig:LCS}
\end{center}
\end{figure}

The dynamics interaction between the macroscopic level
and the microscopic order of molecules are modeled with
(strongly nonlinear) parabolic PDE systems,  involving:
  \begin{itemize}
  \item
  the macroscopic fluid dynamics in velocity-pressure variables of the  Navier-Stokes type,
  \item
  a microscopic order variable modeled by a (vectorial) gradient flow system.
   \end{itemize}
  
 We assume LCs confined in an open bounded domain $\Omega \subset \R^N $(N = 2 or 3) with boundary $\partial\Omega$ which is thermally isolated during the time interval $[0,+\infty)$. Then, the macroscopic dynamic can be described by the velocity-pressure variables $\uu : \Omega \times [0, +\infty) \rightarrow \R^N$ and $p : \Omega \times [0, +\infty)\rightarrow \R$ respectively. For isotropic fluids, these variables are governed by the Navier-Stokes equations, but in LCs the anisotropic microscopic configurations modify the macroscopic dynamics, and reciprocally, the macroscopic movement has an influence on the orientation of molecules.

The mathematical systems related with LCs have been widely study in recent times, where most of these works are devoted to study Nematic phases. A simplified model for Nematic LCs was introduced by Lin in \cite{lin} and studied  by Lin and Liu in \cite{linliu1,linliu2} and by Coutand and Shkoller in \cite{Coutand}.
 A simplified model for Smectic-A LC was proposed by E in \cite{E} and studied by Liu in \cite{liu}, by Climent and Guill\'en in \cite{regul-smectic,cli_gui_14} and by Wu and Segatti in \cite{wu-segatti}. Both are simplified models from the original equations in the continuum theory of liquid crystals due to Ericksen and Leslie, which was developed during the period of 1958 through 1968. We refer the reader to \cite{review_LC} for a detailed review about the main results in the mathematical analysis of these models.

From the numerical analysis and simulations point of view, most of the effort of researchers have been focused on approximating Nematic phases. We recommend the reader \cite{Badia-Guillen-Gutierrez,GK} (and the references therein) as state-of-the-art reviews of numerical schemes to approximate Nematic LCs and other energy-based models.
For Smectic phases, the literature concerning the development of numerical schemes (and the study of their properties) is still very scarce. The main concern of this paper is to fulfill this lack of numerical approximations of Smectic phases, presenting for the first time energy-stable numerical schemes to approximate these systems.

This work is organized as follows: In Section~\ref{sec:SmeA}, we recall the theory of Smectic-A, starting with the static Oseen-Frank's  theory and then detailing the dynamics of these systems. Afterwards, we introduce a reformulation of the system that allows the spatial approximation of the system considering $C^0$ finite elements.

Then, numerical schemes to approximate this reformulation of the system are presented in Section~\ref{sec:generic-scheme}. We first introduce a finite element scheme for the spatial approximation and afterwards we present a generic second order finite differences approximation in time.

We have carried out numerical simulations to study the performance of the proposed schemes. In particular, we have detailed the results of these computations in Section~\ref{sec:num_sim}.
Finally, we state the conclusion of our work in Section~\ref{sec:con}.
\section{Smectic-A Liquid Crystals}\label{sec:SmeA}
\subsection{Static Oseen-Frank's  theory of Smectic-A LCs}
Let $\Om\subset \R^d$ ($d=2,3$) be the bounded domain occupied by
the LC, with boundary $\partial\Om$. The equilibrium states correspond to a
minimum of a (elastic) free energy governing the system. An usual form for this energy is the \emph{Oseen-Frank} free energy:
\beq
E_{ela}(\d,\nabla\d) = \int_\Om \left(\frac{k_1}{2}(\na\cdot\d)^2
+\frac{k_2}{2} (\d\cdot(\na\times\d)) ^2
+\frac{k_3}{2}|\d\times(\na\times\d)|^2\right) d\x,
\eeq
where $\d=\d(x)$ is the director vector and   $k_i>0$ are elastic constants (hereafter, $|\cdot|$ denotes the euclidean norm). A further simplification is the equal
constant  case $k_1=k_2=k_3(=1)$, where $E_{ela}$ reduces to the Dirichlet energy functional:
\beq
  E_{ela}=\frac{1}{2}\int_\Om  (\nabla\cdot\d)^2 d\x.
\eeq
For uniaxial Smectic LCs, the molecules are arranged in almost
incompressible layers (whose normal vector is denoted by $\n$) of almost constant
width. Moreover, each layer consists
of a single optical axis $\d$ perpendicular to the layer.  In particular, for Smectic-A LCs, both unitary vectors coincide 
\beq
\d =\n.
\eeq
On the other hand, due to the incompressibility of the layers
\beq
\nabla\times\n=0,
\eeq
there exists a potential function $\p:\Om\to \R$ such that 
\beq
\n=\nabla\p,
\eeq
 and the level sets of $\p$ will represent the layer structure. Moreover, if the domain $\Om$ has a simply connected boundary $\partial\Om$, then $\p$ can be computed via the Poisson-Neumann problem:
\beq
-\Delta \p = -\na\cdot \n \quad \hbox{in $\Om$},\quad \na\p\cdot\m =\n\cdot\m \quad \hbox{on $\partial\Om$},
\eeq
where 
 $\m$ denotes the exterior normal vector to $\partial\Om$.

Since $\d = \n$,  $\na\times \n=0$ and $\n=\na\p$, the (elastic) Dirichlet
functional can be reduced to
\beq
E_{ela}=\frac{1}{2}\int_\Om (\nabla\cdot\n)^2 d\x=\frac{1}{2} \int_\Om  (\Delta\p)^2 d\x
\eeq
and the static minimization problem has a convex functional but with
a non-convex constraint:
\beq
\min_\p  E_{ela}(\Delta\p)\quad \hbox{subject to}\quad
|\na\p |=1.
\eeq
The optimality system associates to this minimization problem reads:
\beq
\frac{\delta E_{ela}}{\delta \p}:=\Delta^2\p -\na\cdot(\xi \na \p)=0, \quad |\na\p |=1,
\eeq
where $\xi$ is a Lagrange multiplier. In order to relax the non-convex constraint 
$|\na\p |=1$, it is usual to consider the following regularized energy (by considering a
penalization of the Ginzburg-Landau type):
\beq
  E(\p)=E_{ela}(\Delta\p) + E_{pen}(\na\p)
   := \frac{1}{2} \int_\Om|\Delta\p|^2
  + \frac{1}{\eps^2} \int_\Om  F(\na\p),\quad
  F(\n)=\frac{1}{4}(|\n|^2-1)^2.
\eeq
Then the relaxed  minimization problem is now a problem without constraints but for a non-convex functional:
\beq
\min_\p  E(\p).
\eeq
The optimality system associated to this problem reads
\beq
\left\langle
\frac{\delta E(\p)}{\delta\p}
\right\rangle
=\int_{\Om}\Delta\p\Delta\overline\p
+\frac1{\varepsilon^2}\f(\nabla\p)\cdot \na\overline\p=0
\quad\pato\overline\p,
\eeq
where $\f( \n)=\na_\n F(\n)=(|\n|^2-1)\n$. In order to show which boundary conditions are admissible, we need to split the computations into two different steps:
\begin{enumerate}
\item Assume $\Delta\p|_{\partial\Om}=0$ [N1] or $\nabla\overline\p\cdot\m|_{\partial\Om}=0$ [D2], arriving at
\beq\label{aux-var}
\left\langle
\frac{\delta E(\p)}{\delta\p}
\right\rangle
=-\int_{\Om}\w\cdot\nabla\overline\p,
\quad
\w:=\nabla\Delta\p - \frac1{\varepsilon^2}\f(\nabla\p).
\eeq

\item Then, assume $\w\cdot\m|_{\partial\Om}=0$ [N2] or $\overline\p|_{\partial\Om}=0$ [D1], to arrive at
\beq\label{def-EL(phi)}
\frac{\delta E(\p)}{\delta\p}
=\nabla\cdot\w
=\Delta^2\p - \frac1{\varepsilon^2}\nabla\cdot\f(\nabla\p)
\quad
\mbox{ in }
\Om.
\eeq

\end{enumerate}
Therefore, the admissible combinations for boundary conditions are:
\beq\label{eq:possi_BCs}
\left\{\ba{lll}
\mbox{[D1-D2]} & \p|_{\partial\Om}=\p_1, & \nabla\p\cdot\m|_{\partial\Om}=\p_2,
\\ \hueco
\mbox{[D1-N1]} & \p|_{\partial\Om}=\p_1, & \Delta\p|_{\partial\Om}=0,
\\ \hueco
\mbox{[D2-N2]} & \nabla\p\cdot\m|_{\partial\Om}=\p_2, & \w\cdot\m|_{\partial\Om}=0,
\\ \hueco
\mbox{[N1-N2]} & \Delta\p|_{\partial\Om}=0, & \w\cdot\m|_{\partial\Om}=0.
\ea\right.
\eeq
Note that pairs [D1-N2] and [D2-N1] are not admissible boundary conditions in order to get \eqref{def-EL(phi)}.




\subsection{Dynamics of Smectic-A LCs.}
We are interested in the dynamics of Smectic-A LCs confined in the domain $\Omega $ during a time interval $(0,T)$.  

The macroscopic dynamic variables are denoted by $(\u,p)$, that represents the incompressible velocity
and the pressure of the flow, respectively. For the well known case of isotropic
newtonian fluids (assuming constant density), the equilibrium of forces is modeled in terms of these macroscopic
variables  $(\u,p)$, arriving at the PDE system (linear momentum system):
\beq
\frac{ D\u}{ Dt} -\na\cdot \sigma =0,\quad \na\cdot\u=0 \quad
\hbox{in $\Omega\times(0,T)$,}
\eeq
where $\dfrac{ D\u}{ Dt}=\u_t+ (\u\cdot\na)\u$ is the
convective time derivative (the derivative along the streamlines),
 $\sigma$ is the  Cauchy stress tensor given by the so-called Stokes' law:
\beq
\sigma=-p\mathcal{I} +2\nu D(\u), 
\eeq
where $\mathcal{I}$ represents the $N$-dimensional identity matrix and $D(\u)=(\na\u+\na\u^t)/2$ is the deformation tensor (symmetric part of $\na\u$).

  For  Smectic-A LCs, we will consider the model proposed by E \cite{E},
  where the elastic (and  dissipative)
influence of the order vector $\n$ in the linear momentum system is modeled by the following Cauchy stress tensor:
\beq
\sigma=-p\, \mathcal{I} +\sigma^d(D(\u),\n) + \sigma^e(\n),
\quad \n=\na\p, 
\eeq
where
 $\sigma^d=\sigma^d(D(\u),\n)$ is the dissipative (symmetric) stress tensor defined as:
\begin{equation}\label{dissipative-tensor}
 \sigma^d=\mu_1(\n^tD(\u)\n)\n\otimes\n+\mu_4D(\u) +\mu_5(D(\u)\n
\otimes\n+\n\otimes D(\u)\n ),
\end{equation}
and
 $\sigma^e$ is the (nonsymmetric) elastic tensor, defined as:
 \beq
 \sigma^e
 = -\na\p \otimes \w + \Delta\p\na(\na\p),
 \eeq
 where $\w$ was defined in (\ref{aux-var}) and $\otimes$ denotes the tensorial product.
Then, it is possible to derive a PDE system for Smectic-A LC:
\begin{itemize}

\item[$\blacktriangleright$] Conservation of linear momentum derives to the  $(\u,p)$-system:
\beq
\frac{ D\u}{ Dt}  +\na p -\na\cdot (\sigma^d + \lambda\,
\sigma^e)=0,\quad \na\cdot\u=0 \quad \hbox{in $\Omega\times(0,T)$,}
\eeq
where $\lambda>0$ is a constant coefficient coupling the kinetic
and the elastic energy.

\item[$\blacktriangleright$] Conservation of angular momentum derives the $\p$-equation:
\beq
\frac{ D \p}{ Dt} + \gamma \, \frac{\delta E}{\delta \p}=0 \quad \hbox{in $\Omega\times(0,T)$,}
\eeq
where $\gamma>0$ is a constant coefficient (elastic relaxation time) and $\dfrac{\delta E}{\delta \p}$ is defined in (\ref{def-EL(phi)}). Note that this equation can be viewed as a Allen-Cahn type from the phase-field framework.
\end{itemize}
The elastic tensor $\sigma^e(\n)$ must be  chosen in order to assume
that the dynamics is governed by the (dissipative) energy equality:
\beq\label{energy_equality}
\frac{d}{dt}\Big( E_{kin}(\u)+\lambda\, E(\p)  \Big)
+\int_\Om \sigma^d(D(\u),\n) : D(\u) + \lambda\gamma\int_\Om \left|\frac{\delta E}{\delta \p}\right|^2 = 0,
\eeq
where 
$ E_{kin}(\u)=\frac{1}{2}\int_\Om|\u|^2$ is the kinetic energy and 
\beq
\sigma^d(D(\u),\n):D(\u)=\mu_1 (\n^t D(\u) \n)^2+\mu_4 |D(\u)|^2+\mu_5 | D(\u)\n|^2\ge \mu_4 |D(\u)|^2.
\eeq
 Assuming time-independent boundary conditions for $\p$, that is $\p_t|_{\partial\Om}=0$ and $\na\p_t\cdot\m|_
{\partial\Om}=0$, this energy equality is deduced considering 
\beq
\int_\Om (\u\hbox{-system})\cdot \u+ \lambda\int_\Om(\p\hbox{-equation}) \frac{\delta E}{\delta \p},
\eeq
  %
%
 because the following equality holds \cite{E}:
\beq
-\na\cdot\sigma^e = - \frac{\delta E(\p)}{\delta \p} \na\p + \na (E(\p)),
\eeq
hence 
\beq
-\int_\Om (\na\cdot\sigma^e) \cdot \u + \int_\Om \u \cdot  \na\p \, \frac{\delta E}{\delta \p} =0.
\eeq
Accordingly, the differential problem  reads:
$$
\left\{
\begin{array}{rlll}
        \displaystyle\frac{D\u}{Dt}
       -  \na\cdot\sigma^d(D(\u),\n)+\nabla \widetilde p
       - \lambda \,\frac{\delta E}{\delta \p} \nabla\p
         & = & {\boldsymbol 0}&\mbox{ in $\Omega\times(0,T)$,}
         \\ \hueco
        \nabla\cdot{\boldsymbol u} & = &  0&\mbox{ in $\Omega\times(0,T)$,}
        \\ \hueco
	\dfrac{D\p}{Dt}
        + \gamma\,\dfrac{\delta E}{\delta \p} &=& 0 &\mbox{ in $\Omega\times(0,T)$,}     
        \end{array}
        \right.
\leqno{(P)}
$$
where  $\sigma^d(D(\u),\n)$ is the symmetric tensor defined in (\ref{dissipative-tensor}), $\widetilde p=p+\lambda E$ is a modified potential function (that for simplicity it is  renamed as $p$) and the expression of $\dfrac{\delta E}{\delta \p}$ depending on $\phi$ is given in \eqref{def-EL(phi)}. This differential problem must be ended with initial conditions
$$
{\boldsymbol u}_{|t = 0 } ={\boldsymbol u}_0,
        \quad \p_{|t = 0 }  =  \p_0 \quad \mbox{in $\Omega$,}
$$
and the non-slip condition $ {\boldsymbol u}= 0$ on  $\partial\Omega\times(0,T)$ and 
  one admissible  boundary conditions for $\p$ given in \eqref{eq:possi_BCs}. For instance, $\w\cdot\m|_{\partial\Om}=0$ (that is [N2]) implies the conservation property $\dis\dfrac{d}{dt}\int_\Om \p=0$.
 
 Note that  the equilibrium solutions ($\u=0$ and $\p^*$ satisfying $\dfrac{\delta E}{\delta \p}(\p^*)=0$ and the corresponding boundary conditions for $\p^*$) are  equilibrium solutions associated to system $(P)$, where the pressure  reduces to:
$$ \widetilde p= p + \lambda\, E(\p^*)= \rm constant .  $$

The fact that  system $(P)$ satisfies the (dissipative) energy law \eqref{energy_equality}, and therefore the free energy $ E_{kin}(\u) + \lambda\, E (\p)$ has a physical dissipation, play an essential role in the main mathematical results of problem $(P)$, which can be summarized as follows (see \cite{review_LC} for a review on mathematical analysis of Nematic and Smectic-A LCs):
\begin{itemize}

\item[$\blacktriangleright$] ~\cite{liu} Imposing time-independent Dirichlet boundary conditions
  for $\p$, one has
  \begin{enumerate}
\item 
existence of
global in time weak solutions of $(P)$, satisfying the regularity 
\beq
\p\in L^{\infty}(0,T; H^2(\Omega))\cap L^2(0,T;H^4(\Omega)), 
\eeq
\beq
\quad {\bf u}\in L^{\infty}(0,T; \L^2(\Omega))\cap
L^2(0,T;\H^1(\Omega)),
\eeq
and   the energy inequality associated to \eqref{energy_equality} (changing in  \eqref{energy_equality} the equality $=0$ by the inequality $\le 0$).
\item
existence (and uniqueness) of local  in time regular solutions, which is  global in time for large viscosity coefficient  $\mu_4 $  (dominant viscosity case),
\item
convergence  by subsequences at infinite time towards  equilibrium solutions  $(0,\p^\star)$ (with $\p^ \star $ being a solution of $\frac{\delta E}{\delta \p}(\p^\star)=0$).
\end{enumerate}
These results extend the same type of results already obtained for nematic liquid crystal in \cite{linliu1,linliu2}.

\item[$\blacktriangleright$]  ~\cite{regul-smectic} Imposing time-dependent Dirichlet conditions for $\p$, one has existence of time-periodic weak solutions,
which are regular for large viscosity   $\mu_4 $.  

  These results extend the same type of results already obtained for nematic liquid crystal \cite{reprod,regul-nematic}.
  
\item[$\blacktriangleright$]~\cite{wu-segatti}  Imposing periodic boundary conditions for $\p$, one has convergence to a unique equilibrium solution of the whole sequence $(\u(t),\p(t))$ as time $t\to +\infty$.

\item[$\blacktriangleright$]~\cite{cli_gui_14}  Imposing time-independent Dirichlet  conditions for $\phi$, one has convergence to a unique equilibrium solution of the whole sequence $(\u(t),\p(t))$ as time $t\to +\infty$.


\end{itemize}

\subsection{Reformulation as a mixed second order problem}
In this section, we derive a new formulation of system $(P)$ that will allow us to consider numerical schemes in the $C^0$ finite element framework.
For simplicity,  we will consider $2D$ domains and the [D2-N2] boundary conditions for $\p$, although it is not difficult to extend the same arguments to the $3D$ case and other admissible boundary conditions for $\p$, see \eqref{eq:possi_BCs} above.

We start introducing a new unknown
\beq
\psi:=-\Delta\p.
\eeq 
Hence, we can rewrite the elastic energy in terms of $\psi$, i.e.,
\beq
E_{ela}
=\frac{1}{2} \int_\Om  (\Delta\p)^2 d\x
=\frac{1}{2} \int_\Om  \psi^2 d\x.
\eeq
Accordingly, we can write the following reformulation of problem
$(P)$ for unknowns $(\u, p,\p,\psi)$:
$$
\left\{
\begin{array}{rlll}
\displaystyle\frac{D\u}{Dt}
       - \na\cdot\sigma^d(D(\u),\nabla\p)+\nabla p
       + \frac\lambda\gamma \,\Big(\p_t +{\boldsymbol u}\cdot  \nabla\p\Big)\nabla\p
         & = & {\boldsymbol 0}&\mbox{ in $\Omega\times(0,T)$,}
        \\ \hueco
        \nabla\cdot{\boldsymbol u} & = &  0&\mbox{ in $\Omega\times(0,T)$,}
         \\ \hueco\dis
        \frac{1}{\gamma}\Big(\p_t +{\boldsymbol u}\cdot  \nabla\p\Big)
        -\Delta \psi-\dfrac1{\varepsilon^2}\na\cdot\f(\nabla\p) &=&0&\mbox{ in $\Omega\times(0,T)$,}
        \\ \hueco
        -\Delta \p - \psi & = &  0&\mbox{ in $\Omega\times(0,T)$,}
       \\ \hueco
        {\boldsymbol u}= 0,\quad \nabla\p\cdot\m=0, \quad  (\na\p+\dfrac1{\varepsilon^2}\f(\na\p))\cdot\m&=&0,
 &\mbox{ on $\partial\Omega\times(0,T)$,}
        \\ \hueco
        {\boldsymbol u}_{|t = 0 } ={\boldsymbol u}_0,
        \quad \p_{|t = 0 }& = & \p_0&\mbox{ in $\Omega$.}
        \end{array}
        \right.
\leqno{(RP)}
$$


\section{Numerical approximations}\label{sec:generic-scheme}
In this section, we will introduce a numerical scheme to approximate system $(RP)$. In particular we present a $C^0$-finite element approximation in space and a second-order finite difference scheme in time.
\subsection{A generic FEM space-discrete scheme}
Firstly, we will define a space-discrete scheme using $C^0$-finite elements. 
Let $\{{\cal{T}}_h \}_{h>0}$ be a  regular triangulations family  of $\Omega$ with $h$ denoting the mesh size. Then, the unknowns $(\u,p,\p,\psi)$ are approximated by
the conforming  finite element spaces:
\beq
(\U_h, P_h,\Phi_h,\Psi_h)\subset(\H^1(\Omega),L^2(\Omega), H^1(\Omega), H^1(\Omega) )
\eeq
 via the following  mixed variational
formulation of $(RP)$:
Find
\beq
(\u(t),p(t),\p(t),\psi(t))\in
\U_{h0}\times P_{h0}\times  \Phi_h\times  \Psi_{h}
\eeq
such that 
\beq\label{space-scheme}
\left\{\begin{array}{rcl}
\displaystyle\Big(\u_t,\bar\u\Big) 
+c\Big(\u,\u,\bar \u\Big)
+ \Big(\sigma^d(D(\u),\nabla\p),D(\bar\u)\Big)
-\Big(p,\nabla\cdot\bar\u\Big)
+\frac\lambda\gamma \Big(\big(\p_t + \u\cdot\nabla\p\big)\nabla\p,\bar\u \Big)&=&0, 
\\ \hueco
 \Big(\nabla\cdot \u, \bar p\Big) &=& 0 ,
\\ \hueco\dis
\frac1\gamma\Big(\p_t + \u\cdot\nabla\p,\bar \p\Big)
+\Big(\nabla \psi, \nabla \bar\p\Big)
+ \dfrac1{\varepsilon^2}\Big(\f(\nabla\p),\nabla\bar\p\Big)
&=& 0,  
\\ \hueco
\Big(\nabla\p,\nabla \bar \psi \Big)
- \Big(\psi, \bar \psi \Big)
&=&0, 
\end{array}\right.
\eeq
for any 
$ (\bar\u,\bar p,\bar\p,\bar \psi)\in \U_{h0}\times
P_{h0} \times  \Phi_{h}\times \Psi_{h} $,  
where $ \U_{h0}:=\U_h\cap\H^1_0(\Omega)$ and $ P_{h0}:=P_h\cap L^2_0(\Omega)$. Hereafter  $\Big(\cdot,\cdot\Big)$ denotes  the 
  $L^2(\Omega)$-scalar product and  $c\Big(\cdot,\cdot,\cdot\Big)$ is the  trilinear 
antisymmetric form defined by
\beq
c\Big(\u,\v,\w\Big)=\Big((\u\cdot\nabla)\v,\w\Big)+\frac{1}{2}
\Big((\nabla\cdot\u)\,\v,\w\Big)\quad\forall\,
\u,\v,\w\in\U_h.
\eeq
On the other hand, from (\ref{dissipative-tensor}) and the symmetry of $\sigma^d(D(\u),\nabla\p)$, one has
\beq
\ba{c}
 -\Big(\na\cdot\sigma^d(D(\u),\nabla\p),\bar\u\Big)=
\Big(\sigma^d(D(\u),\nabla\p),D(\bar\u)\Big)
\\ \hueco
= \mu_1\Big( (\nabla\p^tD(\u)\nabla\p), (\nabla\p^tD(\bar\u)\nabla\p)\Big)
+\mu_4 \Big (D(\u), D(\bar\u)\Big) 
+2\,\mu_5\Big(D(\u)\nabla\p, D(\bar\u)\nabla\p \Big).
\ea
\eeq

We assume the
following  \textit{inf-sup} stability  condition:
\begin{itemize}

\item[$\blacktriangleright$]  For $(\U_{h0}, P_{h0})$:
  $\displaystyle\sup_{\u\in\U_{h0}}
\frac{(p,\na\cdot\u)}{\|\u\|_{H_0^1}}\ge \beta_1 \|p\|_{L^2} $
 for all $p\in P_{h0}$.


\end{itemize}
The following choices could be considered \cite{girault-raviart}:
\begin{itemize}

\item[$\blacktriangleright$]   $(P_1\rm{-bubble}) \times P_1  \hbox{ for $(\U_h,P_h)$}, 
\quad P_1\times P_1  \hbox{ for $(\Phi_h, \Psi_h)$}.$

\end{itemize}

\begin{lemma}
Each possible  solution
$(\u(t),p(t),\p(t), \psi(t))$ of the
space-discrete scheme \eqref{space-scheme} satisfies the following space-discrete version of the energy law \eqref{energy_equality}:
\begin{equation}\label{induction}
\frac{d}{dt} E_{tot} \Big(\u(t),\nabla\p(t), \psi(t) \Big)
 +  \Big( \sigma^d(D(\u(t)),\nabla\p(t)), D(\u(t)) \Big) 
 + \frac\lambda\gamma \left\| \p_t(t) + \u(t)\cdot\nabla\p(t) \right\|_{L^2}^2 = 0,
\end{equation}
where
\beq
E_{tot}(\u,\nabla\p, \psi):=E_{kin}(\u)+\lambda \Big(E_{ela}(\psi) + E_{pen}(\nabla\p) \Big),
\eeq
with
\beq
E_{kin}(\u)=\displaystyle \frac{1}{2} \|\u\|_{L^2}^2,\quad E_{ela}(\psi)= \frac{1}{2} \| \psi \|_{L^2}^2,\quad E_{pen}(\nabla\p)=\frac 1{\varepsilon^2}\int_{\Omega}F(\nabla\p).
\eeq
\end{lemma}
\begin{proof}
At the initial time we have $\psi(0)=-\Delta\p(0)$, so we can replace equation \eqref{space-scheme}$_4$ by its time derivative
$$
\Big(\nabla\p_t,\nabla \bar \psi \Big) - \Big(\psi_t, \bar \psi \Big)=0,
$$
and then taking $ (\bar\u,\bar p,\bar\p,\bar \psi)
=(\u(t), p(t),\p_t(t) ,\lambda\, \psi(t)) $
 as test functions in the space-discrete scheme \eqref{space-scheme} and  using the  equalities
$$
c\Big( \u(t),\u(t),\u(t) \Big)=0,
$$
and
\begin{equation}\label{deriv-pot}
\Big(\f(\nabla\p(t)),\frac{d}{dt} \nabla\p(t) \Big) =\frac{d}{dt}
\int_\Om F(\nabla\p(t)),
\end{equation}
one arrives at \eqref{induction}.
\end{proof}
\begin{remark}
Since 
$$
\Big( \sigma^d(D(\u(t)),\nabla\p(t)), D(\u(t)) \Big) \ge \mu_4 \int_\Om |D(\u(t))|^2,
$$
from \eqref{induction}, one has the following energy inequality
$$
\frac{d}{dt} E_{tot}\Big(\u(t),\nabla\p(t), \psi(t) \Big)
 +  \mu_4 \| D(\u(t)) \|_{L^2}^2
+ \frac\lambda\gamma \left\| \p_t(t) + \u(t)\cdot\nabla\p(t) \right\|_{L^2}^2
 \le  0.
$$
\end{remark}

\subsection{Generic second order in time scheme}
We assume an uniform partition of $[0, T]$:  $t_n = nk$, with  $k =T/N$ denoting the time step. We consider fully coupled second order in time semi-implicit finite difference schemes,  defined as:
\medskip

Given $(\uu^{n},\p^{n}, \psi ^n)$, compute 
$(\uu^{n+1}, p^{n+\frac12}, \p^{n+1}, \psi ^{n+\frac12}) \in \U_{h0}\times P_{h0}\times \Phi_h \times W_h$
such that for any 
$(\bar\uu, \bar p, \bar\p, \bar \psi) \in \U_{h0}\times P_{h0}\times \Phi_h \times W_h$:
\beq\label{CH4:SMA_generic_scheme}
\left\{\ba{rcl}
 \dis 
 \left(\delta_t\uu^{n+1} , \bar\uu\right) 
 + c \Big(\widetilde\uu,\uu^{n+\frac12},\bar\uu \Big) +
  \Big(\sigma^d(D(\uu^{n+\frac12}),\nabla\widetilde\p) , D(\bar\uu)\Big) &&
  \\ \dis
  -\Big( p^{n+\frac12} , \nabla\cdot \bar\uu \Big) 
  + \frac\lambda\gamma \Big(\big(\delta_t\p^{n+1} + \uu^{n+\frac12}\cdot \nabla\widetilde\p\big)\nabla\widetilde\p, \bar\uu\Big) &=& 0,
\\ \hueco
\Big(\nabla\cdot \uu^{n+\frac12} , \bar p \Big) &=& 0 ,
\\ \hueco
\dis \frac1\gamma\Big(\delta_t\p^{n+1} + \uu^{n+\frac12}\cdot \nabla\widetilde\p, \bar \p \Big) 
+  \Big(\nabla \psi ^{n+\frac12}, \nabla \bar\p \Big) 
 + \dfrac1{\eps^2}\Big(\f^k(\nabla\p^{n+1},\nabla\p^n), \nabla\bar\p \Big)
&=&0,
\\ \hueco
\left(\psi ^{n+1},\bar \psi\right)
-\left(\nabla\p^{n+1},\nabla\bar \psi\right)
 &=& 0 \ea\right.
\eeq
where $\delta_t$ denotes the discrete time derivative ($\delta_t a^{n+1}:= (a^{n+1} - a^n)/k$), unknown  $p^{n+\frac12}$ is an approximation  at midpoint $t^{n+\frac12}:=(t^n+t^{n+1})/2$ (directly computed), while $\uu^{n+\frac12}:=(\uu^{n+1}+\uu^{n})/2$ and $\psi ^{n+\frac12}:=(\psi ^{n+1}+ \psi ^{n})/2$.

To assure the second order accuracy of the previous scheme, $\f^k(\nabla\p^{n+1},\nabla\p^n)$,  $\widetilde\uu$ and $\nabla\widetilde\p$ have to be defined as second order approximations of $\f(\nabla\p(t^{n+\frac12}))$,  $\uu(t^{n+\frac12})$ and $\nabla\p(t^{n+\frac12})$, respectively.

 \begin{lemma}
The following discrete energy inequality holds:
\beq\label{SO_energy_law}
\delta_tE_{tot}(\u^{n+1},\nabla\p^{n+1}, \psi ^{n+1})
+\mu_4\|D(\u^{n+\frac12})\|^2_{L^2(\Omega)}
+\frac\lambda\gamma\big\|\delta_t\p^{n+1} + \uu^{n+\frac12}\cdot \nabla\widetilde\p \big\|^2_{L^2(\Omega)}
+ \dfrac{\lambda}{\varepsilon^2}\, ND_{phobic}^{n+1} \le 0,
\eeq
where 
$$
ND_{phobic}^{n+1}=\int_\Om \f^k(\nabla\p^{n+1},\nabla\p^n)\cdot\delta_t\nabla\p^{n+1}
 - \delta_t \left(\int_\Om F(\nabla\p^{n+1})\right). 
$$
\end{lemma}
\begin{proof}
Considering $\psi ^0=P^{L^2}_{\Psi_{h}}(-\Delta\p^0)$ (where $P^{L^2}_{\Psi_{h}}$ denotes the $L^2$-projection onto $\Psi_{h}$), by induction and using from the previous time step that 
$$
\left(\psi ^{n},\bar \psi\right)-(\nabla\p^{n},\nabla\bar \psi)=0,
$$
we can replace equation \eqref{CH4:SMA_generic_scheme}$_5$ by its discrete time derivative 
$$
\left(\delta_t \psi ^{n+1},\bar \psi\right)
-\left(\delta_t\p^{n+1},\nabla\bar \psi\right)
 = 0 .
$$
Then, taking as test functions in the scheme (\ref{CH4:SMA_generic_scheme})
$$ (\bar\u,\bar p,\bar\p,\bar \psi)
=(\u^{n+\frac12}, p^{n+\frac12},\lambda\,\delta_t \p^{n+1} , \lambda\,\psi^{n+\frac12}),
$$
the term $\left(\delta_t\p^{n+1},\nabla \psi^{n+\frac12}\right)
$ cancels, 
and by using the  expressions
\beq
c\Big( \widetilde\u,\u^{n+\frac12},\u^{n+\frac12} \Big)=0,
\eeq
and
\beq
\Big( \sigma^d(D(\u^{n+\frac12}),\nabla\widetilde\p), D(\u^{n+\frac12}) \Big) \ge \mu_4 \int_\Om |D(\u^{n+\frac12})|^2,
\eeq
one arrives at \eqref{SO_energy_law}.
\end{proof}

\begin{definition}
The  scheme \eqref{CH4:SMA_generic_scheme} is energy-stable if it holds
\beq\label{CH4:SMA_generic_energy_law}
\delta_tE_{tot}(\u^{n+1},\nabla\p^{n+1}, \psi ^{n+1})
 + \mu_4\|D(\u^{n+\frac12})\|_{L^2(\Om)}^2
 +\frac\lambda\gamma\big\|\delta_t\p^{n+1} + \uu^{n+\frac12}\cdot \nabla\widetilde\p\big\|^2_{L^2(\Omega)}
 \le 0.
\eeq
In particular, energy-stable schemes satisfy the energy decreasing in time property,
$$
E_{tot}(\u^{n+1},\nabla\p^{n+1}, \psi ^{n+1}) \le E_{tot}(\u^{n},\nabla\p^{n}, \psi ^{n}), \quad\pato n.
$$
\end{definition}
In particular, depending on the approximation considered of $\f^k(\nabla\p^{n+1},\nabla\p^n)$ we will obtain different numerical schemes, with different discrete energy laws. We refer the reader to \cite{GK} for more detailed information about different ways of efficient handling this term.

\begin{lemma}[Global energy-stability] Assuming the  scheme \eqref{CH4:SMA_generic_scheme} is energy-stable, then the following estimates hold:
$$
\u^{n+1} \hbox{  in } l^2(0,T;\H_0^1),
\quad \psi^{n+1} \hbox{ in }  l^\infty(0,T;\L^2),
\quad\mbox{ and }\quad
\nabla\p^{n+1}\hbox{ in } l^4(0,T;\L^4).
$$
\end{lemma}
\begin{proof}
By induction respect to the time step $n$, \eqref{CH4:SMA_generic_energy_law} implies the following
bounds independent of $k,h,\eps$ (if $|\u^0|^2+ |\psi^0|^2\le C$ and  $\int_\Om F(\nabla\p^0)\le C\, \eps^2$, with $C>0$ independent of $k,h$ and $\eps$):
\beq
E_{tot}(\u^{n+1},\nabla\p^{n+1}, \psi ^{n+1}) \le C,\quad k\sum_n\Big(
 \| D(\u^{n+1}) \|_{L^2(\Om)}^2
+ \big\|\delta_t\p^{n+1} + \uu^{n+\frac12}\cdot \nabla\widetilde\p  \big\|_{L^2(\Om)}^2 \Big) \le C,
\eeq
hence the following $(k,h,\eps)$-independent
estimates  hold
\begin{eqnarray*}
  \u^{n+1} \hbox{ and }  \psi ^{n+1}&\hbox{in}& l^\infty(0,T;\L^2),
   \qquad \frac1{\eps^2}\int_\Om F(\nabla\p^{n+1}) \quad\hbox{in}\quad l^\infty(0,T),\\
 D(\u^{n+1}) &\hbox{in}& l^2(0,T;L^2).
 \end{eqnarray*}
In particular, by applying  Korn's inequality to $\u^{n+1}$, one also has the $(k,h,\eps)$-independent estimate 
$$
 \u^{n+1} \quad\hbox{in}\quad l^2(0,T;\H_0^1).
$$
On the other hand, owing to the inequality $\int_\Om F(\nabla\p) \ge \big(\|\nabla\p\|^4_{L^4(\Om)}-|\Om|\big)$, one also has the following $(k,h)$-independent  estimate
$$
 \nabla\p^{n+1}\hbox{ is bounded  in } l^4(0,T;\L^4),
$$
although this bound depends on $\eps$.
\end{proof}

\subsubsection{How to define $\f^k(\nabla\p^{n+1},\nabla\p^{n})$}

There are several possible ways of approximating potential $\f(\nabla\phi)$ (check \cite{GK} for a detailed review). Indeed, in the last years many works from different physical applications have appeared in the literature, presenting new ways of dealing with this kind of potentials.
For the Cahn-Hilliard equation, implicit approximations have been often considered \cite{ELLIOT_FRE1,ELLIOT_FRE2,ELLIOT1,FENG2,COPETTI,HECTOR2} where a Newton method is usually employed in order to compute the nonlinear scheme. There is an implicit-explicit approximation of the potential that does not introduce any numerical dissipation which have been widely used in phase field models 
\cite{DU,FENG,HUA,HYON_KWAK_LIU} and in the Liquid Crystals context \cite{LIN}.
On the other hand, many authors split the potential into a convex and a non-convex part in order to assure the existence of some numerical dissipation to obtain a unconditional energy-stable scheme \cite{EYRE,Becker-Feng-Prohl,HU,KKL,WISE,FILHO,HECTOR}, although the resulting schemes are nonlinear. 
The idea of splitting the potential has been also considered for thin film epitaxy \cite{ref:Xiaoming}. Moreover, some linear schemes have been studied in \cite{SHEN-YANG1, SKJ, KG, KG2, WU_ZWI_ZEE}.

Solvability of scheme \eqref{CH4:SMA_generic_scheme} will depend on the approximation of the potential
term  $\f^k(\nabla\p^{n+1},\nabla\p^{n})$. In the case of energy-stable schemes (non-linear schemes), solvability follows from the discrete energy law \eqref{SO_energy_law} and either an application of the Brouwer fixed-point theorem (cf. Corollary 1.1 of \cite{girault-raviart}) or an application of the Leray-Schauder fixed-point theorem in a finite dimensional setting. But, the uniqueness  of these nonlinear schemes is not clear in general, depending on the behavior of 
\beq
\frac{\lambda}{\eps^2 k}
\Big(\f^k(\nabla\p^{n+1},\nabla\p^{n})-\f^k(\nabla \widetilde\p^{n+1},\nabla\p^{n}),
\nabla\p^{n+1} -  \nabla \widetilde\p^{n+1} \Big)
\eeq
where $\nabla\p^{n+1}$ and  $\nabla \widetilde\p^{n+1}$ are two possible solutions of the scheme. In fact, using a convex-concave semi-implicit approximation (assuming implicitly the convex part and explicitly the concave one) the uniqueness is deduced from the monotony of $\f^k(\cdot,\nabla\p^{n})$ (because the convex part of the potential is treated implicitly). In the case of the concave part treated implicitly, one can not deduce unconditional uniqueness and a constraint of time step $k$ small enough must be imposed.

\subsubsection{How to define $\widetilde\uu$ and $\nabla\widetilde\p$}\label{sec:how}
The scheme \eqref{CH4:SMA_generic_scheme} has been designed as second order approximations of the associated model, but it is necessary to define $\widetilde\uu$ and $\nabla\widetilde\p$ as second order approximations of $\uu(t^{n+\frac12})$ and $\n(t^{n+\frac12})$, respectively. We propose two different ways of dealing with these terms:
\begin{enumerate}
\item The first one consists on choosing a Crank-Nicolson approach,
\begin{equation} \label{CH4:C-N}
\widetilde\uu=(\uu^{n+1} + \uu^n)/2\quad \mbox{ and }\quad \dis\nabla\widetilde\p=(\nabla\p^{n+1} + \nabla\p^n)/2.
\end{equation}
In this case we obtain a nonlinear one-step scheme. 
\item The second way consists on using a \textbf{BDF2} approximation for each one of the unknowns, i.e., 
\begin{equation} \label{CH4:BDF2}
\widetilde\uu=(3\uu^n - \uu^{n-1})/2\quad \mbox{ and } \quad \nabla\widetilde\p=(3\nabla\p^n - \nabla\p^{n-1})/2.
\end{equation}
Then, using these two-step approximations, the linearity and the solvability conditions of the schemes will depend just on the choice of the potential term $f^k(\nabla\p^{n+1},\nabla\p^n)$.
\end{enumerate}

\section{Numerical simulations}\label{sec:num_sim}

In this section we study the behavior of the numerical schemes presented through the paper. In particular, we will focus on the second order scheme obtained using the generic scheme \eqref{CH4:SMA_generic_scheme} with the linear potential approximation \textbf{OD2} (whose efficiency have been studied for the Cahn-Hilliard equation in \cite{GK}): 
\begin{equation}
\label{f-od1}
 \f^k(\nabla\p^{n+1},\nabla\p^{n})=\f(\nabla\p^{n}) + \frac12  \f'(\nabla\p^{n})(\nabla\p^{n+1}-\nabla\p^{n}),
\end{equation}
that it is a second order in time approximation of the potential $\f(\nabla\p(t^{n+\frac12}))$. In particular, for the Ginzburg-Landau potential $\f(\nabla\phi)$ considered in this paper, a direct computation yields to
\beq\label{OD2}
 \f^k(\nabla\p^{n+1},\nabla\p^{n}) =  (\nabla\p^{n}\cdot \nabla\p^{n+1})\nabla\p^{n} 
  + |\nabla\p^{n}|^2 \frac{\nabla\p^{n+1}-\nabla\p^{n}}{2} -  \frac{\nabla\p^{n+1}+\nabla\p^{n}}{2} .
\eeq
By using ideas from \cite{KG,KG2}, the unique solvability of the scheme \eqref{CH4:SMA_generic_scheme} using the approximation (\ref{OD2}) can be assured under the constraint $k<2\varepsilon^2/\gamma$. Nevertheless, energy-stability of scheme \eqref{CH4:SMA_generic_scheme} and \eqref{OD2}  (possibly assuming constraints of type $k$ small enough  in function of $h$ and $\eps$),  remain as an open problem.

In order to complete a fully linear scheme, we take the linear \textbf{BDF2} approximations presented in \eqref{CH4:BDF2}.

The domain considered is $\Omega=(-1,1)^2$ with an uniform space partition of size $h=1/32$,  the time step set to $k=10^{-5}$ and 
the physical parameters set to  $\mu_1 = \mu_4 = \mu_5 = \lambda = \gamma = 1$
and $ \varepsilon = 0.05$. A finite element approximation is considered in space, using the software \textit{FreeFem++} \cite{Hetch} for carrying out the simulations under the following choices for the discrete spaces:
$$
(\u,p)\sim P_2 \times P_1\quad\mbox{ and }\quad (\p, \psi)\sim  P_1\times P_1.
$$ 
We consider the initial conditions
$$
\p(0) = \sin(x)\cos(y)^2,
\quad \psi(0)=-\Delta\p(0),
\quad \uu_1(0) = 0, 
\quad \uu_2(0) = 0,
 $$
the boundary conditions 
$$
\uu=\nabla\p\cdot \m=(\nabla \psi+\dfrac1{\varepsilon^2}\f(\na\p))\cdot \m = 0 \quad \mbox{ on $\partial\Omega$}
$$ 
(which correspond to \mbox{[D2-N2]} in \eqref{eq:possi_BCs})
and we compute a first step $(\n^1,\u^1)$ using the scheme \eqref{CH4:SMA_generic_scheme} with the Crank-Nicolson approximation described in \eqref{CH4:C-N}, i.e.,
$$
\widetilde\uu=(\uu^{n+1} + \uu^n)/2\quad \mbox{ and }\quad \dis\widetilde\n=(\n^{n+1} + \n^n)/2,
$$
in order to be able to use the \textbf{BDF2} approximation for $\widetilde\uu$ and $\widetilde\n$ in the rest of time iterations.


The dynamics of the layer $\phi$ and the velocity field $\u$ are presented in Figures~\ref{fig:DYN1}-\ref{fig:DYN2} while in Figures~\ref{fig:KIN}-\ref{fig:ELA} we plot the evolution of the kinetic and elastic energy. The dynamics consists on the deformation of the layer configuration to reach an equilibrium configuration, that induces movement in the fluid part. Moreover, the total energy dissipates until there are no changes  on the layer and on the velocity field, i.e., the system is reaching an equilibrium configuration.

\begin{figure}[H]
\begin{center}
\includegraphics[width=1.0\textwidth]{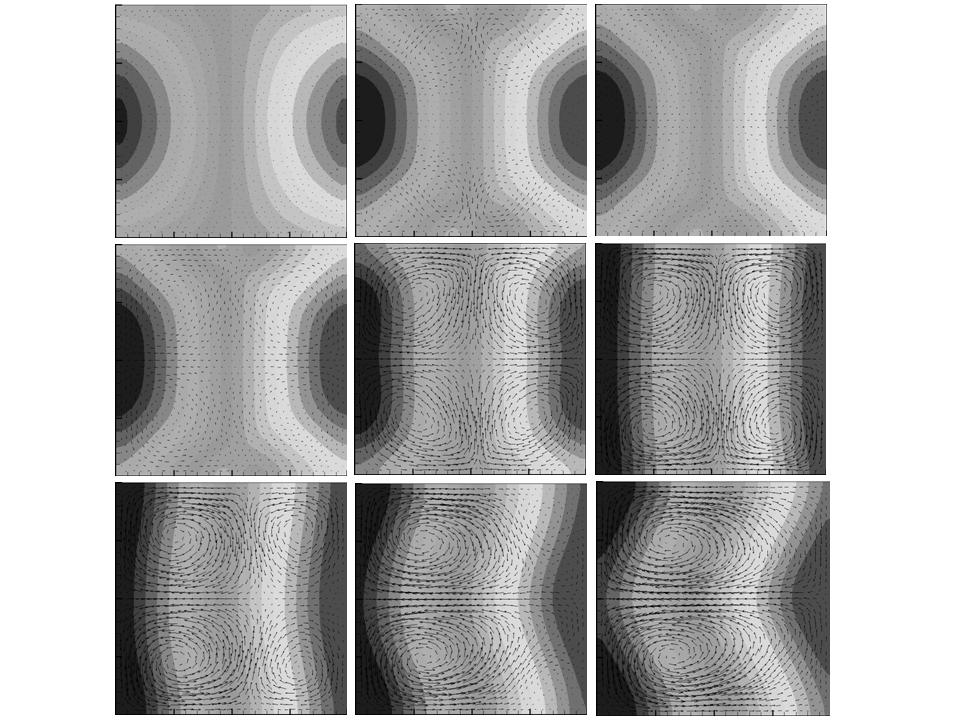}
\caption{From left to right and up to down, the evolution of $\phi$ in time $t=0.00001,\,0.005,\,0.01,\,0.015,\,0.02,\,0.025,\,0.03,\,0.035$ and $0.04$}\label{fig:DYN1}
\end{center}
\end{figure}

\begin{figure}[H]
\begin{center}
\includegraphics[width=0.75\textwidth]{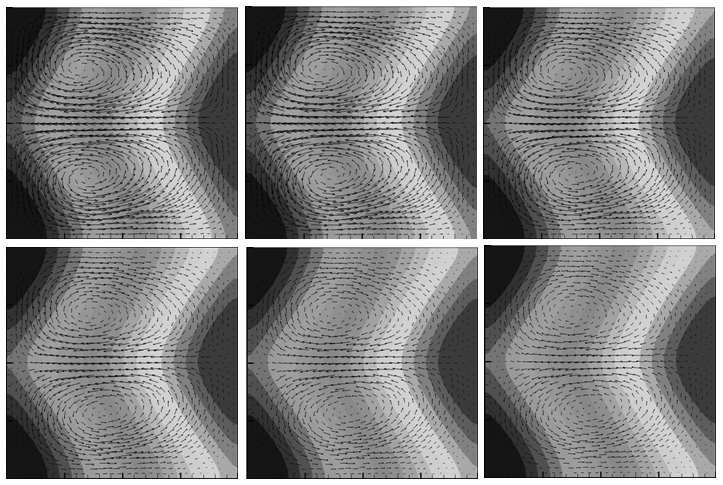}
\caption{From left to right and up to down, the evolution of $\phi$ in time $t=0.045,\,0.05,\,0.06,\,0.07,\,0.08$ and $0.086$}\label{fig:DYN2}
\end{center}
\end{figure}

\begin{figure}[H]
\begin{center}
\includegraphics[width=0.9\textwidth]{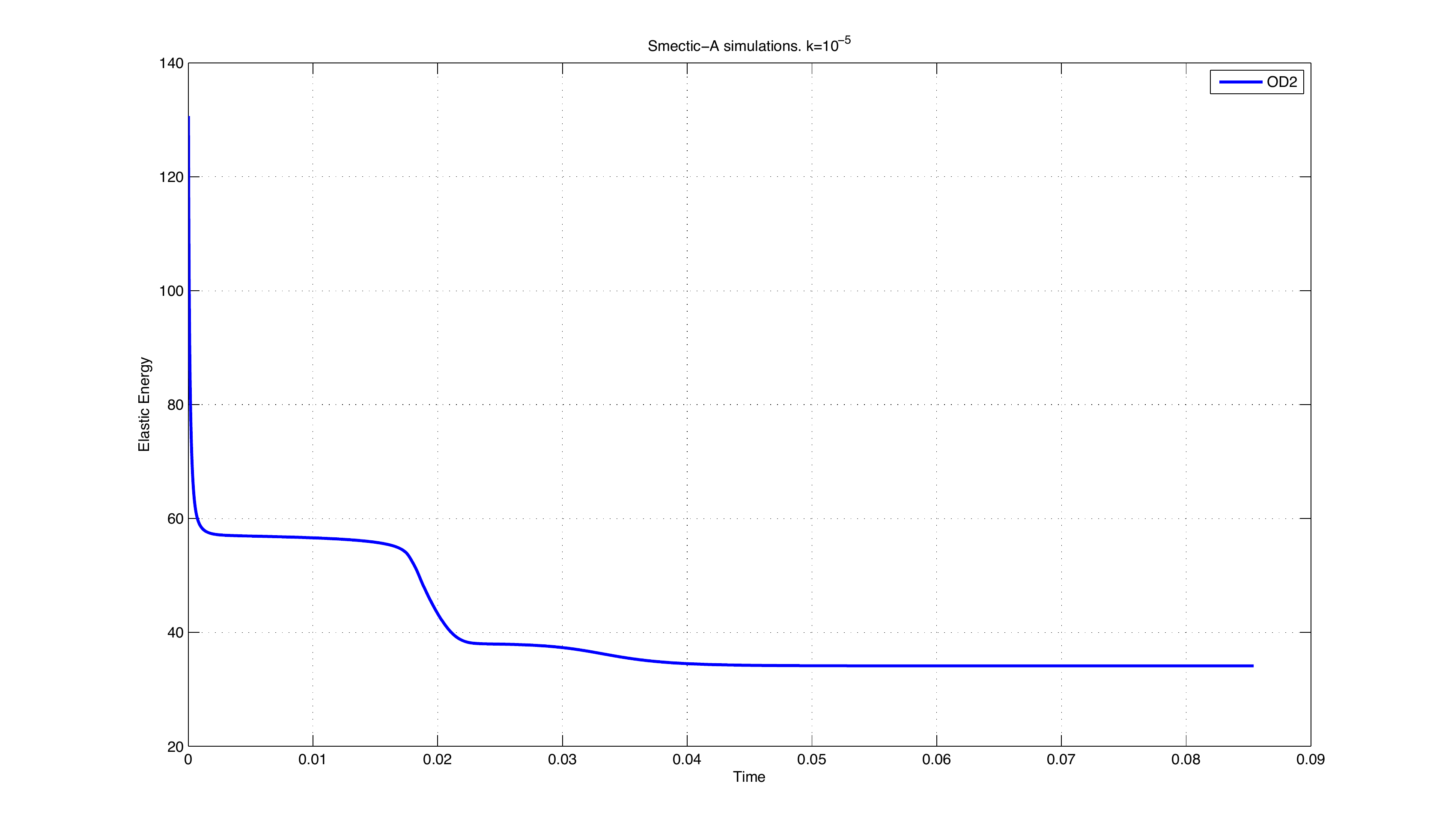}
\caption{Kinetic energy}\label{fig:KIN}
\end{center}
\end{figure}

\begin{figure}[H]
\begin{center}
\includegraphics[width=0.9\textwidth]{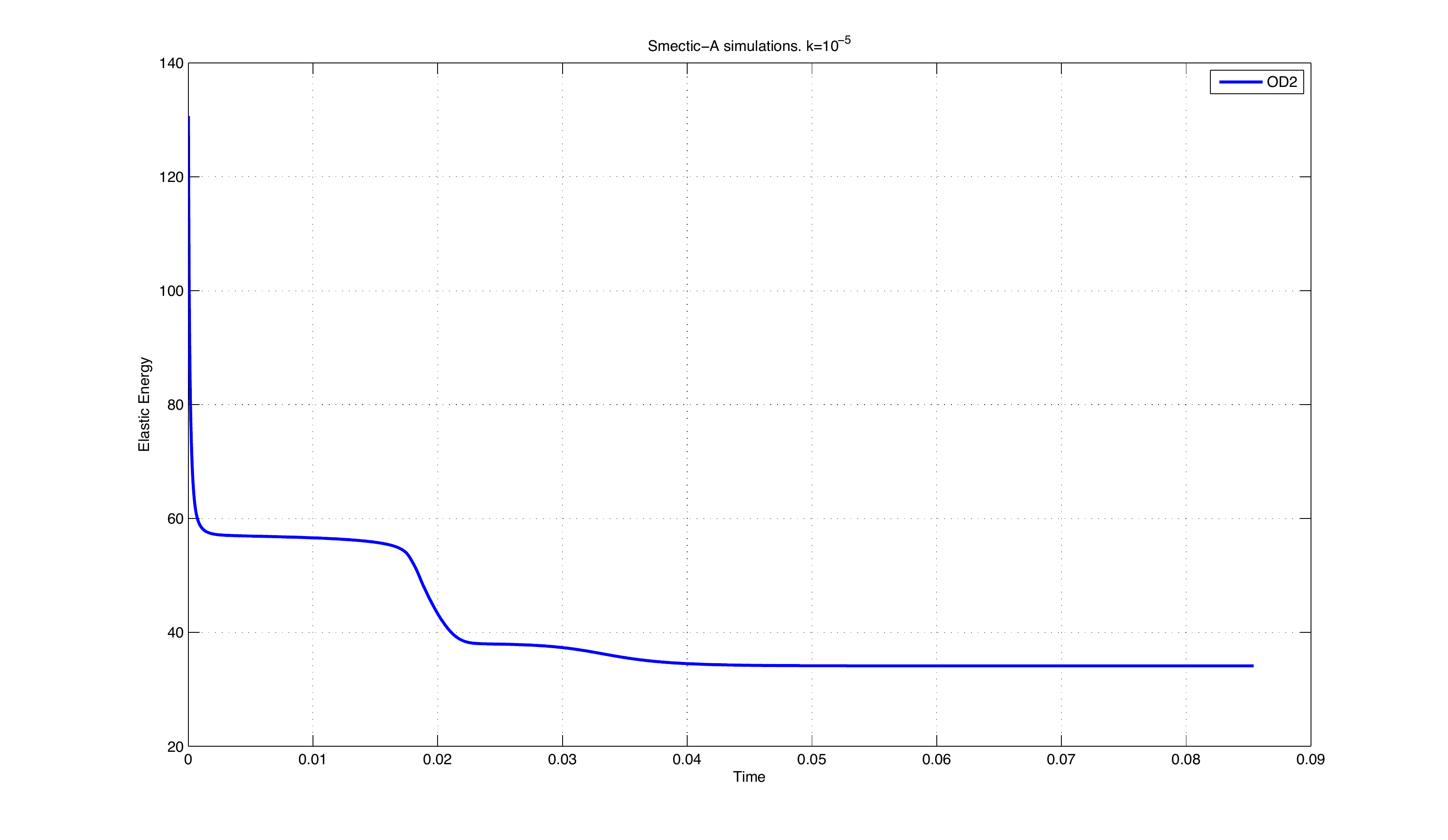}
\caption{Elastic energy}\label{fig:ELA}
\end{center}
\end{figure}

\section{Conclusions}\label{sec:con}

In this paper, we have presented a new reformulation of the Smectic-A liquid crystals system, where a new unknown $\psi=-\Delta\p$ have been introduced in order to arrive at a mixed second-order problem. This new formulation allows us to recover a dissipative energy law, that is in correspondence with the energy law associated to the original problem.

We approximate this new formulation using second-order finite differences in time and by $C^0$-finite elements in space. For this scheme, we deduce a discrete version of the dissipative energy law derived in the continuous problem. 

Finally, numerical simulations are reported to show that the proposed  scheme capture the dynamics of Smectic-A liquid crystals. More extensive numerical tests and studies of other effects, such as the influence of physical parameters and the  interaction with other type of fluids will be also investigated in the future.

\subsection*{Acknowledgements.} 
This research has been  
partially supported by MTM2012-32325 (Ministerio de Econom\'ia y Competitividad, Spain). Giordano Tierra has also been partially supported by ERC-CZ project LL1202 (Ministry of Education, Youth and Sports of the Czech Republic).

\end{document}